\documentclass[10pt]{article}
\usepackage[utf8]{inputenc}
\usepackage{amsmath, amsfonts,amsthm, fullpage}
\usepackage{wasysym}
\usepackage{graphicx}
\usepackage{xcolor}
\usepackage{ulem}
\usepackage{algorithm}
\usepackage{algpseudocode}
\usepackage[colorlinks, citecolor=blue, unicode]{hyperref}

\newcommand{\specnorm}[1]{\Vert  #1 \Vert_{2\rightarrow 2}}

\DeclareFontFamily{U}{matha}{\hyphenchar\font45}
\DeclareFontShape{U}{matha}{m}{n}{
      <5> <6> <7> <8> <9> <10> gen * matha
      <10.95> matha10 <12> <14.4> <17.28> <20.74> <24.88> matha12
      }{}
\DeclareSymbolFont{matha}{U}{matha}{m}{n}
\DeclareFontSubstitution{U}{matha}{m}{n}

\DeclareFontFamily{U}{mathx}{\hyphenchar\font45}
\DeclareFontShape{U}{mathx}{m}{n}{
      <5> <6> <7> <8> <9> <10>
      <10.95> <12> <14.4> <17.28> <20.74> <24.88>
      mathx10
      }{}
\DeclareSymbolFont{mathx}{U}{mathx}{m}{n}
\DeclareFontSubstitution{U}{mathx}{m}{n}

\DeclareMathDelimiter{\vvvert}{0}{matha}{"7E}{mathx}{"17}

\newcommand{\gennorm}[1]{\vvvert #1 \vvvert}

\newtheorem{definition}{Definition}
\newtheorem{remark}{Remark}
\newtheorem{lemma}{Lemma}
\newtheorem{theorem}{Theorem}
\newtheorem{corollary}{Corollary}
\begin{document}



\title{Robust recovery of low-rank matrices and low-tubal-rank tensors from noisy sketches\thanks{A.M. was supported by U.S. Air Force Grant FA9550-18-1-0031, NSF
Grant DMS 2027299, and the NSF+Simons Research Collaborations on the
Mathematical and Scientific Foundations of Deep Learning. Y.Z. was partially supported by  NSF-Simons Research Collaborations on the Mathematical and Scientific Foundations of Deep Learning. Y.Z. also acknowledges support from NSF DMS-1928930 during his participation in the program “Universality and Integrability in Random Matrix
Theory and Interacting Particle Systems” hosted by the Mathematical Sciences Research Institute
in Berkeley, California, during the Fall semester of 2021.}}

\author{Anna Ma\thanks{Department of Mathematics, University of California Irvine 
  ({anna.ma@uci.edu}).}
\and Dominik Stöger\thanks{Department of Mathematics and Mathematical Institute for Machine Learning and Data Science (MIDS), KU Eichst\"att-Ingolstadt 
  ({dominik.stoeger@ku.de}).}
\and Yizhe Zhu\thanks{Department of Mathematics, University of California Irvine 
  ({yizhe.zhu@uci.edu}).}}

\maketitle

\begin{abstract} A common approach for compressing large-scale data is through matrix sketching. In this work, we consider the problem of recovering low-rank matrices from two noisy linear sketches using the double sketching scheme discussed in Fazel et al. (2008), which is based on an approach by Woolfe et al. (2008). Using tools from non-asymptotic random matrix theory, we provide the first theoretical guarantees characterizing the error between the output of the double sketch algorithm and the ground truth low-rank matrix. We apply our result to the problems of low-rank matrix approximation and low-tubal-rank tensor recovery.
\end{abstract}



\section{Introduction}

The prevalence of large-scale data in data science applications has created immense demand for methods that can aid in reducing the computational cost of processing and storing said data. Oftentimes, data, such as images, videos, and text documents, can be represented as a matrix, and thus, the ability to store matrices efficiently becomes an important task. One way to efficiently store a large-scale matrix $X_0$ $\in \mathbb{C}^{n_1 \times n_2}$ is to store a \textit{sketch} of the matrix, i.e., another matrix $Y \in \mathbb{R}^{m_1 \times m_2}$ such that two goals are accomplished. Firstly, the sketch of $X_0$ must be cheaper to store than $X_0$ itself, i.e., we want $m_1m_2 \ll n_1n_2$. Second, 
we must be able to recover $X_0$ from its sketch, meaning that we are able to reconstruct $X_0$ exactly or at least be able to obtain a high-quality approximation of $X_0$.

A variety of works have been produced for the setting in which $X_0$ is a low-rank matrix, and one wishes to recover $X_0$ using its sketches, see, e.g., \cite{fazel2008compressed,woolfe2008fast,clarkson2009numerical,romberg_sketching, tropp_sketching, tropp_sketching2}. However, in certain settings, one may only have access to noisy sketches. For example, noise in the sketch can arise
from measurements or numerical round-off errors, a low-rank approximation of the ground truth matrix, or be a result of sketching noisy matrices~\cite{widrow2008quantization, bhunia2022sketching}. Noisy sketches have also appeared in recovering low-rank matrices from noisy linear observations \cite{negahban_wainwright,srinivasa2019decentralized}. 
It should be noted that in the aforementioned noisy observation models, only one noisy sketch is obtained, while in the model we consider here~\eqref{eq:sketchstep}, there are two noisy sketches available. Noisy sketches are also used in differentially private algorithms for matrix factorization \cite{upadhyay2018price,arora2018differentially}. In particular, the model we consider in \eqref{eq:sketchstep} is closely related to the output perturbation method described in \cite{upadhyay2018price,arora2018differentially}, which is an important intermediate step for algorithms that preserve differential privacy of a low-rank matrix factorization (see Algorithm 1 in \cite{upadhyay2018price}).

In this work, we analyze the noisy double-sketch algorithm considered but not theoretically studied in \cite{fazel2008compressed} and which is based on an approach by Woolfe et al. \cite{woolfe2008fast}. Here, the term ``double-sketch" refers to the first algorithm in \cite{fazel2008compressed} where two independent random sketching matrices are applied to $X_0$ by multiplying from the left and right, respectively. We show that when the sketching matrices are i.i.d. complex Gaussian random matrices, one can recover the original low-rank matrix $X_0$ with high probability where the error on the approximation depends on the noise level for both sketches. Here, we do not assume that one has access to the exact rank of $X_0$ but instead, an approximate rank $r$, which satisfies $r > r_0 := \text{rank} (X_0)$. We also remark on the utility of our theoretical guarantees for when the double sketch algorithm is used not for low-rank matrix recovery but instead for low-rank approximation with noise.   Lastly, we present results for the application of this work to a more extreme large-scale data setting in which one wants to recover a low-tubal-rank tensor \cite{kilmer2013third} (see Definition \ref{def:low_tubal}), which is used as a model for multidimensional data in applications.

A key step in our robust recovery analysis is to control the perturbation error of a low-rank matrix under noise.
A standard way is to apply Wedin's Theorem, or Davis-Kahan Theorem \cite{wedin1972perturbation,davis1970rotation,chen2020spectral,o2018random}, which results in a bound that depends on the condition number of the low-rank matrix. However, our proof is based on an exact formula to calculate the difference between output and the ground truth matrix and a  detailed analysis of the random matrices (extreme singular value bounds of Gaussian matrices \cite{vershynin2018high,rudelson2010non,tao2010random} and the least singular value of truncated Haar unitary matrices \cite{banks2020pseudospectral,dumitriu2012smallest}) involved in the double sketch algorithm. This novel approach yields a bound independent of the condition number of the low-rank matrix (Theorem \ref{thm:main}). Due to the Gaussian structure of our sensing matrices, our results are non-asymptotic, and all the constants involved in the probabilistic error bounds are explicit.

\subsection{Low-rank matrix reconstruction}
The double sketch method, which is analyzed in this paper, was considered (but not theoretically studied) in \cite{fazel2008compressed} to recover low-rank matrices from noisy sketches and is based on an approach by Woolfe et al. \cite{woolfe2008fast}. This method is a single-pass method, meaning that the input matrix needs only to be read once.
Related single-pass sketching methods were analyzed in \cite{clarkson2009numerical,woolfe2008fast,tropp_sketching,tropp_sketching2} in the noise-free scenario, meaning that, in contrast to the model studied in this paper, the sketches are not perturbed by additive noise.

The single-pass model is motivated by the fact that in many modern applications, one cannot make several passes over the input matrix $X_0$. However, if one is able to make several passes over the input matrix $X_0$, one can, in many applications, expect superior performance compared to single-pass algorithms. We note that multi-pass algorithms for low-rank matrix approximation using sketches have been studied in, e.g., \cite{zhou2012bilateral,halko2011finding,PCA_largedatasets}. 

The so-called problem of compressive PCA was studied in \cite{romberg_sketching} and \cite{Azizyan_compressivePCA}.
It can be interpreted as a variant of sketching where only the columns of a matrix are sketched.
However, this problem is not directly comparable to the setting in the paper at hand, as in compressive PCA, a different sketching matrix is used for each column.
\subsection{Low-tubal-rank tensor recovery}

The notion of a low-tubal-rank tensor stems from the t-product, originally introduced by \cite{kilmer2011factorization}. We state the relevant definitions for order-3 tensors. We refer to \cite{kilmer2013third} for a more comprehensive treatment. More general definitions for tensors of higher orders can be found in \cite{kernfeld2015tensor,lu2019tensor}.

\begin{definition}[Operations on tensors, see, e.g., \textcolor{black}{\cite[Section 3.2]{kilmer2011factorization}}] 
Let $\mathcal{A} \in \mathbb{C}^{n_1 \times n_2 \times n_3}$. The unfold of a tensor is defined to be the frontal slice stacking of that tensor. In other words,
\begin{equation*} 
    \text{unfold}({\mathcal{A}}) = \begin{pmatrix}
        A_1\\
        A_{2}\\
        \vdots\\
        A_{n_3}\\
        \end{pmatrix} \in \mathbb{C}^{n_1 n_3 \times n_2},
\end{equation*}
where $A_i = \mathcal{A}_{:,:,i}$ denotes the $i^{th}$ frontal slice of $\mathcal{A}$.
We define the inverse of the $\text{unfold}(\cdot)$ as $\text{fold}{(\cdot)}$ so that $\text{fold}(\text{unfold}(\mathcal{A})) = \mathcal{A}$.
The block circulant matrix of $\mathcal{A}$ is:
\begin{equation*} 
    \text{bcirc}({\mathcal{A}}) = \begin{pmatrix}
        A_1 & A_{n_3} & A_{n_3-1} & \ldots & A_{2}\\
        A_{2} & A_{1} & A_{n_3} & \ldots & A_{3}\\
        \vdots & \vdots & \vdots & \ddots & \vdots \\
        A_{n_3} & A_{n_3-1} & A_{n_3-2} & \ldots & A_{1}\\
        \end{pmatrix} \in \mathbb{C}^{n_1 n_3 \times n_2 n_3}.
\end{equation*}
\label{def:tensorops}

The conjugate transpose of a tensor $\mathcal A\in \mathbb C^{n_1\times n_2\times n_3}$ is the  tensor $\mathcal A^* \in \mathbb C^{n_2\times n_1\times n_3} $
obtained by conjugate
transposing each of the frontal slices and then reversing the
order of transposed frontal slices $2$ through $n_3$.
\end{definition}

\begin{definition}[Tensor t-product, see \textcolor{black}{\cite[Section 3.2]{kilmer2011factorization}}] Let $\mathcal{A} \in \mathbb{C}^{n_1\times \ell \times n_3}$ and $\mathcal{B}\in \mathbb{C}^{\ell\times n_2 \times n_3}$ then the t-product between $\mathcal{A}$ and $\mathcal{B}$, denoted $\mathcal {A} * \mathcal{B}$, is a tensor of size $n_1 \times n_2 \times n_3$ and is computed as:
\begin{equation}
    \mathcal{A} * \mathcal{B} = \text{fold}(\text{bcirc}(\mathcal{A})\text{unfold}(\mathcal{B})).
\end{equation}
\label{def:tprod}
\end{definition}
In Definition~\ref{def:tprod} the t-product utilizes a block circulant matrix. Since block circulant matrices can be block diagonalized by the Fourier transform, it is useful to define the mode-3 FFT, which allows us to compute the t-product efficiently in the Fourier domain and avoid the need to explicitly form, store, and multiply $\text{bcirc}(\mathcal{A})$ and $\text{unfold}(\mathcal{B})$. See~\cite{kilmer2013third} for more details.

\begin{definition}[Mode-3 fast Fourier transformation (FFT), see \textcolor{black}{\cite[Section 2]{kilmer2013third}}] The mode-3 FFT of a tensor $\mathcal{A}$, denoted $\widehat{\mathcal{A}}$, is obtained by applying the discrete Fourier Transform matrix, $F \in \mathbb{C}^{n_3 \times n_3}$, to each $ \mathcal{A}_{i,j,:}$ of $\mathcal{A}$:
\begin{equation}
 \widehat{\mathcal{A}}_{i,j,:} = F\mathcal{A}_{i,j,:}.
    \label{eq:mode3fft}
\end{equation}
Here,  $F$ is a unitary matrix, $\mathcal{A}_{i,j,:}$ is an $n_3$-dimensional vector, and the product is the usual matrix-vector product. Similarly, the inverse mode-3 fast Fourier transform is defined by applying the inverse discrete Fourier Transform matrix to each $ \widehat{\mathcal{A}}_{i,j,:}$ of $\widehat{\mathcal{A}}$.
\label{def:mode3fft}
\end{definition}

\begin{definition}[t-SVD, see \textcolor{black}{\cite[Theorem 4.1]{kilmer2011factorization}}] The Tensor Singular Value Decomposition (t-SVD) of a tensor $\mathcal{M} \in \mathbb{C}^{n_1 \times n_2 \times n_3}$ is given by 
\begin{equation}
\mathcal{M} = \mathcal{U} * \mathcal{S} * \mathcal{V}^*,
    \label{eq:tsvd}
\end{equation}
where $\mathcal{U} \in \mathbb{C}^{n_1 \times n_1 \times n_3}$ and $\mathcal{V} \in \mathbb{C}^{n_2 \times n_2 \times n_3}$ are unitary tensors and $\mathcal{S} \in \mathbb{R}^{n_1 \times n_2 \times n_3}$ is a tubal tensor (a tensor in which each frontal slice is diagonal), and $*$ denotes the t-product.
\label{def:tsvd}
\end{definition}
 
\begin{definition}[Tubal rank, see \textcolor{black}{\cite[Section 4.1]{kilmer2013third}}]\label{def:low_tubal}
The tubal rank of a tensor $\mathcal{M} = \mathcal{U} * \mathcal{S} * \mathcal{V}^*$ is the number of non-zero singular tubes of $\mathcal{S}$, and it is uniquely defined.
\end{definition}

\begin{definition}[CP rank]
The CP rank of an order three tensor $\mathcal M$ is the smallest integer $r$ such that $\mathcal M$ is a sum of $r$ rank-1 tensors:
\[  \mathcal M=\sum_{i=1}^r u_i\otimes v_i\otimes w_i,
\]
where $u_i\in \mathbb C^{n_1}, v_i\in \mathbb C^{n_2}, w_i\in \mathbb C^{n_3}$, $1\leq i\leq r$.
\end{definition}

If a tensor $\mathcal{M}$ has CP rank $r$, then its tubal rank is at most $r$, see \cite[Remark 2.3]{zhang2016exact}.

\begin{definition}[Tensor Frobenius norm, see, e.g., \textcolor{black}{\cite[Definition 2.1]{kilmer2011factorization}}]
Let $\mathcal M\in \mathbb C^{n_1\times n_2\times n_3}$. The Frobenius norm of $\mathcal M$ is given by 
\[ \|\mathcal M\|_F=\left(\sum_{i_1,i_2,i_3} |\mathcal M_{i_1,i_2,i_3}|^2 \right)^{1/2}.
\]
\end{definition}

Other low-rank tensor sketching approaches have been proposed for low-CP-rank tensors \cite{hao2020sparse} and low-Tucker-rank tensors \cite{sun2020low}.  In the following, we focus on low-tubal-rank tensors since this is the topic of this paper.

Related work considered the recovery of low-tubal-rank tensors through general linear Gaussian measurements of the form $y = A \text{vec} (\mathcal{X}) $~\cite{lu2018exact,zhang2020rip,zhang2021tensor}.
This can be seen as a generalization of the low-rank matrix recovery problem \cite{fazel_recht} to low-tubal-rank tensors.
The proof of tensor recovery under Gaussian measurements in \cite{lu2018exact,zhang2020rip,zhang2021tensor} relies crucially on the assumption that the entries of the measurement matrix $A$ are i.i.d. Gaussian. In this setting, it was shown that the tensor nuclear norm is an atomic norm, and a general theorem from \cite[Corollary 12]{chandrasekaran2012convex} for i.i.d. measurements for atomic norms was used to establish recovery guarantees.
In \cite{wang2021generalized}, a non-convex surrogate for the tensor nuclear norm was proposed and studied.

An extension of the matrix sketching algorithm in \cite{tropp_sketching2} to a low-tubal-rank approximation of tensors was considered in \cite{qi2021t}. However, their setting does not cover noisy sketching, which is the topic of this paper. 
Streaming low-tubal-rank tensor approximation was considered in \cite{yi2021effective}.

\bigskip 

\paragraph{Notations}  We define a standard complex Gaussian random variable $  Z$ as $ Z=  X+i  Y$, where $  X\sim N(0,1/2),   Y\sim N(0,1/2)$ and $  X,  Y$ are independent.
By $\sigma_i (A)$, we define the $i$-th largest singular value of a matrix $A$. 
By $\sigma_{\min} (A)$, we denote the smallest non-zero singular value of a matrix $A$.
 $\specnorm{A}$ is the spectral norm of a matrix $A$, and $\gennorm{A}$ is a general norm of $A$. The matrix $A^*$ is the complex conjugate, and $A^{\dagger}$ is the pseudo-inverse of $A$. Let $U$ be a matrix with orthonormal columns. By $U_{\perp}$,  we denote the orthogonal complement of $U$, which means the column vectors of $U$ and the column vectors of $U_{\perp}$ form a complete orthonormal basis.

\paragraph{Organization of the paper} The rest of the paper is organized as follows. We provide more background on low-rank matrix recovery from sketches in Section \ref{sec:background}. We state our main results and corollaries of robust recovery of low-rank matrix and low-tubal-rank tensors in Section \ref{sec:main}. The proofs of all theorems and corollaries are provided in Section \ref{sec:proof}, and auxiliary lemmas are provided in the Appendix. In Section~\ref{sec:experiments}, we provide numerical experiments on synthetic data that support our theoretical findings and an experiment using real-world CT scan data as an application of double sketching for tensors.

\section{Problem formulation and background}\label{sec:background}
In the following, we denote by $X_0 \in \mathbb C^{n_1\times n_2}$ the target matrix with rank $r_0$ to be sketched.
For that, we use two sketching matrices $S \in \mathbb C^{r\times n_1}$ and $ \tilde{S}\in \mathbb C^{r\times n_2}$, which are assumed to be two independent complex Gaussian random matrices with $r\geq r_0$. 
We consider the model
\begin{align}\label{eq:sketchstep}
     Y:=SX_0+Z, \quad \tilde{Y} :=\tilde{S}X_0^*+\tilde{Z},
 \end{align}
where $Z\in \mathbb C^{r \times n_2}$ and $ \tilde{Z}\in \mathbb C^{r \times n_1}$ represent additive noise.
This model has first been considered (but not been analyzed) in \cite{fazel2008compressed}.
The matrices $S X_0$ and $\tilde{S}X_0^*$, represent the noiseless sketches, which contain information about the row space and column space of the target matrix $X_0$.
Our goal is to estimate the true matrix $X_0$ given the two noisy sketches $Y$ and $\tilde{Y}$.
For that purpose, we use the double-sketch method as considered in \cite{fazel2008compressed}, which outputs
\begin{equation}
X=\tilde{Y}^*(S\tilde{Y}^*)^{\dagger}Y.
\label{eq:doublesketch}
\end{equation}
The matrix $X$ represents an estimate of the target matrix $X_0$.

Note that formula \eqref{eq:doublesketch} has also been studied in \cite[Theorem 4.7]{clarkson2009numerical} in the context of low-rank approximation, i.e., $Z=0$, $\tilde{Z}=0$ and $X_0$ is approximately low-rank. 
(See also Section \ref{sec:lowrankapproximation}, where we apply our results to the setting of noisy low-rank matrix approximation.)
In \cite[Section 4]{tropp_sketching}, an algebraically equivalent, but more numerical stable reformulation of equation \eqref{eq:doublesketch} was described, again for the problem of low-rank matrix approximation. We outline the main ideas of this approach below. Moreover, a similar approach as formula \eqref{eq:doublesketch} has been studied in \cite{woolfe2008fast}. Namely, in that paper instead of \eqref{eq:doublesketch} the formula $X=\tilde{Y}_{r_0}^*(S\tilde{Y}_{r_0}^*)^{\dagger} Y_{r_0} $ has been used, where by $Y_{r_0}$, respectively $\tilde{Y}_{r_0} $, we denote the truncated $r_0$-term singular value decomposition of $Y$, respectively $\tilde{Y}$.

In the noiseless case when $Z=0$, $\tilde{Z}=0$, we denote the output of \eqref{eq:doublesketch} as $\overline{X}$. In this case, the output will read as
\begin{align}\label{eq:output_nonoise}
    \overline{X}=X_0\tilde S^* (SX_0 \tilde {S}^*)^{\dagger} SX_0.
\end{align}

While this work is concerned with the theoretical analysis of the double-sketch method \eqref{eq:doublesketch}, we want to make a few comments regarding
its practical implementation. Our presentation will be based on \cite{tropp_sketching}, which we refer to for more details.
First of all, note that in many applications,  the matrix $X_0$ is too large to be stored on a computer.
A typical scenario, in which the sketching method is used, is a streaming model (see \cite{muthukrishnan2005data}). That is, we have
    \begin{equation*}
    X_0= H_1 + H_2 + H_3 + H_4 +\ldots,
    \end{equation*}
    where the updates $H_i$ must be discarded after they have been processed.
    Typically the matrices $H_i$  possess a simple structure, such as being rank-one. This means that in this scenario, we can obtain the sketches $ SX_0 $ and $\tilde{S} X_0^* $ without ever needing to compute or store the full matrix $X_0$. Note that the sketching method presented in this paper is particularly suited for scenarios where we can only make one pass over the data as in the streaming model.
        
    In order to compute $X$ from formula \eqref{eq:doublesketch} one might obtain the matrix $(S\tilde{Y}^*)^{\dagger}Y $ via its associated least-squares problem.
    However, the matrix $S\tilde{Y}^*$ might be poorly conditioned, leading to an inaccurate solution.
    As suggested in \cite{tropp_sketching}, one can avoid this by considering the unitary-triangular decomposition $ \tilde{Y}^* =QR$, where $Q \in \mathbb{C}^{n_1 \times r}$.
    Then, (if $\tilde{Y}$ has full rank and $SQ \in \mathbb{C}^{r \times r}$ is invertible) we observe that \eqref{eq:doublesketch} is equivalent to
    \begin{equation*}
     X=Q(SQ)^{\dagger}Y.   
    \end{equation*}
    Note that $SQ$ is typically better conditioned than $S\tilde{Y}^*$, 
    implying that $(SQ)^{\dagger}Y$ can be computed more accurately than the corresponding term in~\eqref{eq:doublesketch}. The pseudo-code for this method is provided in Algorithm~\ref{alg:ds_matrix}.
\begin{algorithm}
\caption{Matrix Recovery from Double-Sketch}\label{alg:ds_matrix}
\begin{algorithmic}
\Require Sketches $Y$, $\tilde{Y}$, and sketching matrix $S$ from~\eqref{eq:sketchstep} 
\Ensure $X \in\mathbb{R}^{n_1 \times n_2}$ approximation of target matrix $X_0$
\Function{DoubleSketchMatrixRecovery}{$Y$, $\tilde{Y}$, $S$} 
\State Compute the QR Decomposition of $\tilde{Y}$: $\tilde{Y}^* = QR $ where $Q \in \mathbb{R}^{n_1 \times r}$ and $R \in \mathbb{R}^{r \times r}$
\State $X = Q(SQ)^\dagger Y$
\State \Return $X$
\EndFunction
\end{algorithmic}

\end{algorithm}

The work of \cite{tropp_sketching} suggests, in addition to using the $QR$-decomposition of the sketch $\tilde{Y}^*$, to also use more samples for the co-range sketch $Y$ than for the range sketch $\Tilde{Y}$ to improve the condition number of the matrix $SQ$. In other words, it is suggested that the matrix $Y$ should have more columns than the matrix $\Tilde{Y}$.
While we do not consider this modification in the paper at hand, exploring its noise robustness might be an interesting avenue for future work.
        
In the end, let us stress that in order to access individual entries of $X$, it suffices to compute the corresponding inner products between $Q$ and $(SQ)^{\dagger}Y$. 
In particular,  there is no need to compute or store the complete matrix $X$.

\section{Main results}\label{sec:main}

\subsection{Low-rank matrix recovery from noisy sketches}
In our first result, we show that without noise~\eqref{eq:doublesketch} exactly recovers $X_0$, i.e., $\overline{X}=X_0$ with probability 1. 

\begin{theorem}[Exact recovery] \label{thm:exact}
Let $S\in \mathbb C^{r \times n_1}, \tilde{S}\in \mathbb C^{r\times n_2}$ be two independent complex standard Gaussian random matrices. Furthermore, let $X_0\in \mathbb C^{n_1\times n_2}$ be a  matrix with rank $r_0$. If $r \geq r_0$ and $Z=0$, $\tilde{Z} = 0$ then with probability one $\overline{X}=X_0$,
where $\overline{X}$ is as defined in \eqref{eq:output_nonoise}.
\end{theorem}

Our Theorem \ref{thm:exact} generalized the exact recovery result  \cite[Lemma 6]{fazel2008compressed}, where $r$ is assumed to be exactly $r_0$. Our Theorem \ref{thm:exact} implies the exact value of $r_0$ is not needed for the double sketch algorithm, and one can always use the parameter $r\geq r_0$. In fact, our robust recovery result (Theorem \ref{thm:main}) suggests choosing a larger $r$ makes the output of the double sketch algorithm more robust to noise.

When $Z,\tilde{Z}$ are not all zero, the robust recovery guarantee is given as follows.

\begin{theorem}[Robust recovery]\label{thm:main}
Let $S\in \mathbb C^{r \times n_1}$, $\tilde{S}\in \mathbb C^{r \times n_2}$ be two independent standard complex Gaussian matrices. Let  $Z  \in \mathbb C^{r \times n_2}$ be any matrix, and  $\tilde{Z} \in \mathbb C^{r \times n_1}$ be a matrix such that $(\tilde{S},\tilde{Z})$ is independent of $S$, and $\tilde{Y}=\tilde{S}X^*_0 + \tilde{Z}$ is almost surely of rank $r$ with $r_0<r< n_1$. For any $\delta_1,\delta_2,\epsilon>0$ such that $1 > \delta_2 > \exp \left(- \left( \sqrt{r} - \sqrt{r_0} \right)^2 \right) $ and $\epsilon <1$, with probability at least $1-\delta_1-\delta_2-\epsilon$ over the randomness of $\tilde{S}$ and $S$, the output $X$ from Algorithm \ref{alg:ds_matrix} satisfies
\begin{align*}
\gennorm{X - X_0 } 
    &\leq  \frac{\sqrt{r(n_1-r)} \gennorm{\tilde{Z}}}{\sqrt{\delta_1}(\sqrt{r}-\sqrt{r_0}-\sqrt{\log(1/\delta_2)})} + \frac{\sqrt{r} \gennorm{Z}}{\sqrt{\log(1/(1-\epsilon))}},
\end{align*}
where $ \gennorm{\cdot}$ is any matrix norm that satisfies for any two matrices $A$ and $B$, $ \gennorm{AB} \leq \specnorm{A} \gennorm{B}$, and $ \gennorm{A}=\gennorm{A^*}$. 
In particular, it holds for $\specnorm{\cdot}$ and $\|\cdot \|_F$.  
\end{theorem}

  The condition  that $(\tilde{S},\tilde{Z})$ is independent of $S$  is important in our proof to apply the least singular value bound for truncated Haar random matrices, see Lemma \ref{lem:unitary_truncation}.
    The condition that $\tilde{Y}$ is of rank $r$ can be easily verified in different settings. For example, it holds when $\tilde{Z}$ is Gaussian noise independent of $\tilde{S}$, or $\tilde{Y}=\tilde{S} (X_0^*+\hat{Z})$, where $X_0^*+ \hat{Z}$ is of full rank.

Our proof of Theorem \ref{thm:main} works for $r=r_0$ or $r=n_1$, but the error bounds are slightly different.
\begin{itemize}
    \item When $r=r_0<n_1$, \eqref{eq:minSV0} is replaced by the estimate that
$
    \sigma_{\min}(\tilde{S}V_0)\geq \sqrt{\log(1/(1-\delta_2))} r^{-1/2}
$
with probability $1-\delta_2$. This implies with probability at least $1-\delta_1-\delta_2-\epsilon$,
\begin{align*}
\gennorm{X - X_0 } 
    &\leq  \frac{r\sqrt{n_1-r} \gennorm{\tilde{Z}}}{\sqrt{\delta_1  (\log(1/(1-\delta_2)) }} + \frac{\sqrt{r} \gennorm{Z}}{\sqrt{\log(1/(1-\epsilon))}}.
\end{align*}
\item When $r=n_1$, $U_A$ defined in \eqref{eq:SVD_of_A} is a unitary matrix and $S$ is invertible. We find $\tilde{X}=X_0$ in \eqref{eq:tildeX-X_0} and following the rest of the proof, we obtain with probability at least $1-\epsilon$, 
\[  \gennorm{X - X_0 } \leq \frac{\sqrt{n_1} \gennorm{Z}}{\sqrt{\log(1/(1-\epsilon))}}.
\]
Note that in the case when $r=n_1$, the error is independent of $\tilde{Z}$. This is because in \eqref{eq:doublesketch}, when $S$ is invertible, from part (1) in  Lemma \ref{lem:pinverse}, the expression  becomes \[X=\tilde{Y}^*{(\tilde{Y}^*)}^{\dagger}S^{-1} Y=S^{-1} Y,\] which is independent of $\tilde{Z}$.
\end{itemize}

\subsection{Low-rank matrix approximation from noisy sketches}\label{sec:lowrankapproximation}

Note that Theorem \ref{thm:main} uses the assumption that the rank of the ground-truth matrix $X_0$ is exactly $r_0$. However, in most practical scenarios the matrix $X_0$ is only approximately  low-rank, i.e., the remaining singular values decay quickly to zero.
Indeed, there is a large literature focusing on single-pass sketching algorithms for low-rank approximation (see, e.g., \cite{woolfe2008fast,tropp_sketching,tropp_sketching2}).
However, to the best of our knowledge, none of these works consider the scenario that there is additive noise on the sketched matrices.

In the following, we describe how one can utilize Theorem \ref{thm:main} to also obtain a bound for the approximation error for the scenario that $X_0$ is approximately low-rank and that there is additive noise. In fact, when $X_0$ is not low-rank, we can write $X_0=X_1 + E$, where $X_1$ is the best rank-$r_1$ approximation of $X_0$. Letting $r_0>r > r_1$, we can use the noisy double sketch model in \eqref{eq:sketchstep} to consider the sketches 
\begin{align*}
  Y &= SX_0+Z= SX_1 + (SE+ Z), \\
  \tilde{Y}&= \tilde{S}X_0^*+\tilde{Z}= \tilde{S}X_1^* + ( \tilde{S}  E^* +\tilde{Z}). 
\end{align*}

By interpreting $SE+ Z$ and $\tilde{S}  E^* +\tilde{Z}$ as perturbations of the ``true" sketches $ SX_1 $ and $ \tilde{S}X_1^*  $, we can apply Theorem \ref{thm:main}, which yields Corollary \ref{cor:lowrankapprox}.
The proof of Corollary \ref{cor:lowrankapprox} is given in Section \ref{sec:proof_cor}.

 \begin{corollary}[Low-rank approximation with noisy sketches]\label{cor:lowrankapprox}
For $X_0\in \mathbb C^{n_1\times n_2}$  and  an integer $r_1$ with $r_1<r<n_1$, consider the sketches 
 \[  Y=SX_0+Z, \quad \tilde{Y}=\tilde{S}X_0^*+\tilde{Z},
   \]
where $S \in \mathbb{C}^{r \times n_1}$ and $\tilde{S}\in \mathbb{C}^{r \times n_2}$ are two independent complex Gaussian random matrices and $Z \in \mathbb{C}^{r \times n_2}$ and $\tilde{Z} \in \mathbb{C}^{r \times n_1}$ are noise matrices.
Moreover, assume that $S$ is independent of $\tilde{Z}$.
Recall that the output of Algorithm \ref{alg:ds_matrix} is given by 
\begin{equation*}
     X=\tilde{Y}^*(S\tilde{Y}^*)^{\dagger}Y.
\end{equation*}
 For any $\delta_1, \delta_2,\epsilon>0$ such that $1\ge\delta_2>\exp(-(\sqrt{r}-\sqrt{r_1})^2)$ and $\epsilon <1$, with probability at least $1-\delta_1-3\delta_2-\epsilon$ and when $\tilde{Y}$ is of rank $r$, the output $X$  satisfies
\small \begin{align*}
 &\specnorm{X-X_0}\leq  \sigma_{r_1+1} \left( X_0 \right)  \\
   & +\sigma_{r_1+1}\left( X_0 \right) \left(\frac{\sqrt{r(n_1-r)}(\sqrt{r}+\sqrt{n_2}+\sqrt{\log(1/\delta_2)}) }{\sqrt{\delta_1}(\sqrt{r}-\sqrt{r_1}-\sqrt{\log(1/\delta_2)})} + \frac{\sqrt{r} \left(\sqrt{r}+\sqrt{n_1}+\sqrt{\log(1/\delta_2)}\right)}{\sqrt{\log(1/(1-\epsilon))}} \right)\\
   &+\frac{\sqrt{r(n_1-r)} \specnorm{\tilde{Z}}}{\sqrt{\delta_1}(\sqrt{r}-\sqrt{r_1}-\sqrt{\log(1/\delta_2)})} + \frac{\sqrt{r} \specnorm{Z}}{\sqrt{\log(1/(1-\epsilon))}}.
 \end{align*}
 \end{corollary}

\begin{remark}
When $Z=0$ and $\tilde{Z}=0$ and $X_0$ is approximately low-rank, in \cite{woolfe2008fast} and \cite{tropp_sketching}, two different algorithms have been proposed to approximate the low-rank matrix $X_0$ using the two sketches $Y$ and $\tilde{Y}$.
In the noiseless setting, Corollary \ref{cor:lowrankapprox} below yields a weaker error bound when compared to the bounds in \cite{woolfe2008fast} and \cite{tropp_sketching}. However, our result can handle noise in the two sketches $Y$ and $\tilde{Y} $, while the proofs from \cite{woolfe2008fast,tropp_sketching} are not applicable.
In fact, the proofs in \cite{woolfe2008fast,tropp_sketching} heavily rely on the assumption $Z=0$ and $\tilde{Z}=0$ to use properties of orthogonal projections, which only hold in this noiseless scenario. See
for example \cite[Fact A.2]{tropp_sketching}.
\end{remark}

The error bound in Corollary~\ref{cor:lowrankapprox} is written as a sum of three terms. The first term $\sigma_{r_1+1}(X_0)$ is the error from the best rank-$r_1$ approximation of  {$X_0$}, denoted by $X_1$ in the proof of Corollary~\ref{cor:lowrankapprox}. The second term comes from comparing $X_1$ with the noiseless sketch of $X_0$. The third term is the error from the additive noise $Z$ and $\tilde{Z}$ in the noisy sketch of $X_1$. When the rank of $X_0$ is $r_1$, the first two error terms are $0$, and Corollary~\ref{cor:lowrankapprox} reduces to Theorem~\ref{thm:main}.
The bound in Corollary \ref{cor:lowrankapprox}  is true for any $r_1<r$. Therefore one can optimize $r_1$ to find the best bound in terms of the failure probability and the approximation error.

\subsection{Application to sketching low-tubal-rank tensors}

\begin{algorithm}[t]
\caption{Tensor Recovery from Double-Sketch}\label{alg:ds_tensor}
\begin{algorithmic}[0]
\Require Sketches $\mathcal{Y}$, $\widetilde{\mathcal{Y}}$, and sketching tensor $\mathcal{S}$ from~\eqref{eq:tensordoublesketch} 
\Ensure $\mathcal{X} \in\mathbb{R}^{n_1 \times n_2 \times n_3}$ approximation of target tensor $\mathcal{X}_0$
\Function{DoubleSketchTensorRecovery}{$\mathcal{Y}$, $\widetilde{\mathcal{Y}}$, $\mathcal{S}$} 
\State Compute $\hat{\mathcal{Y}}$, $\widehat{\widetilde{\mathcal{Y}}}$, $\hat{\mathcal{S}}$ via mode-3 FFT  \Comment{See Definition~\ref{def:mode3fft}.}
\For{$i \in \{1,\cdots, n_3 \}$}
\State $\widehat{\mathcal{X}}_i=\textsc{DoubleSketchMatrixRecovery}(\widehat{\mathcal{Y}}_i,\widehat{\widetilde{\mathcal{Y}}}_i, \widehat{\mathcal{S}}_i)$  \Comment See Algorithm~\ref{alg:ds_matrix}.
\EndFor
\State Compute $\mathcal{X}$ from $\widehat{\mathcal{X}}$ via the mode-3 inverse FFT
\State \Return $\mathcal{X}$
\EndFunction
\end{algorithmic}
\end{algorithm}

The approach set forth in \eqref{eq:doublesketch} can be used to sketch and recover low-tubal-rank tensors. For such an application, one considers the low-tubal-rank tensor $ \mathcal{X}_0 \in \mathbb{R}^{n_1 \times n_2 \times n_3}$ with tubal rank $r_0$. Taking the mode-3 FFT of $ \mathcal{X}_0$, one obtains $  \hat{\mathcal{X}}_0 $ which is composed of a collection of $n_3$ matrices (frontal slices) of dimension $n_1 \times n_2$ with rank at most $r_0$. As such, \eqref{eq:doublesketch} can be used to sketch each of the $n_3$ frontal slices of $ \mathcal{X}_0$. Corollary~\ref{cor:tensors} presents guarantees for the approximation error and the proof of Corollary~\ref{cor:tensors} is provided in Section \ref{sec:proof_cor}. 

The recovery of target tensor $\mathcal{X}_0$ from double sketches~\eqref{eq:tensordoublesketch} is outlined in Algorithm~\ref{alg:ds_tensor}. In Corollary~\ref{cor:tensors}, we consider sketching tensors $\mathcal{S}$ and $\mathcal{\tilde{S}}$ such that $\mathcal{S}_1 = S$, $\mathcal{\tilde{S}}_1 = \tilde{S}$ and $\mathcal{S}_k =\textbf{0}$, $\mathcal{\tilde{S}}_k = \textbf{0}$ for all $k \in \{2, \dots, n_3\}$ where $S$ and $\tilde{S}$ are complex standard Gaussian matrices. It should be noted that $\widehat{\mathcal{S}}$ in Algorithm~\ref{alg:ds_tensor} need not be computed and stored explicitly for this choice of $\mathcal{S}$. This is because for any tensor $\mathcal{M}$ such that $\mathcal{M}_k = 0$, for $k \in \{ 2, ..., n_3\}$, $\widehat{\mathcal{M}}_k = \mathcal{M}_1$ for all $k \in \{ 1, ..., n_3 \}$. Thus, instead of needing to compute and store $\mathcal{\widehat{S}}$, one simply needs to store $\mathcal{S}_1$, the first frontal slice of the sketching tensor.

\begin{corollary}[Recovering low tubal-rank tensors]\label{cor:tproduct_sketch}
Let $\mathcal{X}_0 \in \mathbb{C}^{n_1\times n_2 \times n_3}$ be a low-tubal-rank tensor with rank $r_0$. Furthermore, let $S\in \mathbb C^{r\times n_1}, \tilde{S}\in \mathbb C^{r\times n_2}$ be two independent complex standard Gaussian random matrices
and $ \mathcal{Z} \in \mathbb{R}^{ r \times n_2 \times n_3} $ and $ \mathcal{\tilde{Z}}  \in \mathbb{R}^{ r \times n_1 \times n_3} $ be two noise tensors such that $\mathcal{\tilde{Z}}  $ is stochastically independent of $ S $.
Consider the measurements 
\begin{align}
    \mathcal{Y} &= \mathcal{S} * \mathcal{X}_0 + \mathcal{Z}, \,\,\, \mathcal{\tilde{Y}} = \mathcal{\tilde{S}} * \mathcal{X}_0^* + \mathcal{\tilde{Z}}, 
    \label{eq:tensordoublesketch}
\end{align}
where $\mathcal{S}_1 = S$, $\mathcal{\tilde{S}}_1 = \tilde{S}$ and $\mathcal{S}_k =\textbf{0}$, $\mathcal{\tilde{S}}_k = \textbf{0}$ for all $k \in \{2, \dots, n_3\}$. Furthermore, consider the estimate $\mathcal{X}$ of the target tensor $\mathcal{X}_0$ defined as:
\begin{equation}
    \mathcal{X} = \tilde{\mathcal{Y}}^* * (\mathcal{S} * \tilde{\mathcal{Y}}^*)^\dagger *\mathcal{Y},
    \label{eq:tensorX}
\end{equation}
where $\mathcal{M}^\dagger$ denotes the Moore-Penrose pseudo-inverse of the tensor $\mathcal M$~\cite{jin2017generalized}.
 \begin{enumerate}
    \item\textit{(Exact recovery)} If $\mathcal{Z} =\textbf{0} $ and $\mathcal{\tilde{Z}} = \textbf{0}$ then with probability 1, $\mathcal{X} = \mathcal{X}_0$.
    \item\textit{(Robust recovery)} If $r_0 < r < n_1$ and for all $k \in [n_3]$,
$\hat{\tilde{\mathcal{Y}}}_k$ is of rank $r$, then for any $\delta_1, \delta_2,\epsilon>0$ such that $1 > \delta_2>\exp(-(\sqrt{r}-\sqrt{r_0})^2)$ and $\epsilon<1$, with probability at least $1-(\delta_1+\delta_2 +\epsilon)n_3$,
\begin{align}
    \|\mathcal{X} - \mathcal{X}_0 \|_F^2 \leq 
    \frac{2{r(n_1-r)}\|{\tilde{\mathcal{Z}}}\|_F^2}{{\delta_1}(\sqrt{r}-\sqrt{r_0}-\sqrt{\log(1/\delta_2)})^2}  + \frac{2{r}\|{\mathcal{Z}}\|_F^2}{{\log(1/(1-\epsilon))}}.
    \label{eq:robustrecoverytensor}
\end{align}
\item \textit{(Low tubal-rank approximation with noisy sketches)} 
Let $r_1$ be an integer and define the error tensor $\mathcal{E}$ such that $\mathcal{X}_0 = \mathcal{X}_1 + \mathcal{E}$ where $\mathcal{X}_1$ is the best rank $r_1$ approximation of $\mathcal{X}_0$. If $r_1 < r < n_1$ and for all $k \in [n_3]$, $\hat{\tilde{\mathcal{Y}}}_k$ is of rank $r$, then for any $\delta_1, \delta_2, \epsilon > 0$ such that $1 \geq \delta_2 > \exp{-(\sqrt{r} - \sqrt{r_1})^2}$ and $\epsilon < 1$, with probability at least $1 - (\delta_1 - \delta_2 - \epsilon)n_3-2\delta_2$, 
\small\begin{align*}
    &\|\mathcal{X} - \mathcal{X}_0 \|_F^2 \leq 2\|\mathcal{E} \|^2_F  \cdot  \\
   & \left(\frac{4r(n_1-r)\left (\sqrt{r}+\sqrt{n_2}+\sqrt{\log(1/\delta_2)}\right)^2 }{{\delta_1}\left(\sqrt{r}-\sqrt{r_1}-\sqrt{\log(1/\delta_2)}\right)^2} + \frac{4{r} \left(\sqrt{r}+\sqrt{n_1}+\sqrt{\log(1/\delta_2)}\right)^2}{{\log(1/(1-\epsilon))}} +1 \right)\\
   &+\frac{{8r(n_1-r)} \|{\tilde{\mathcal{Z}}\|^2_F}}{{\delta_1}(\sqrt{r}-\sqrt{r_1}-\sqrt{\log(1/\delta_2)})^2} + \frac{{8r} \|{\mathcal{Z}}\|_F^2}{{\log(1/(1-\epsilon))}}. 
   \end{align*}
   \normalsize
\end{enumerate}
\label{cor:tensors}
\end{corollary}

The low tubal-rank approximation with noisy sketches error bound in Corollary~\ref{cor:tensors} depends on $\| \mathcal{E} \|^2_F$ where $\mathcal{E} = \mathcal{X}_0 - \mathcal{X}_1$ and $\mathcal{X}_1$ is the best tubal rank $r_1$ approximation of $\mathcal{X}_0$.
It is well known that the t-product satisfies a tensor Eckart-Young Theorem, i.e., the $k$-truncated t-SVD produces the best tubal rank $k$ approximation of a given matrix~\cite{kilmer2011factorization}. Thus the error term depends on the tail norms of the singular tubes of each frontal slice of $\mathcal{X}_0$. In particular, we have that 
\begin{align}
    \| \mathcal{E} \|^2_F &= \sum_{i=r_1 +1}^{\min(n_1, n_2)} \| {\mathcal{S}}_{i,i,:}\|^2,
\end{align}
where $\mathcal{X}_0 = \mathcal{U} * \mathcal{S}{*} \mathcal{V}^{{*}}$ is the t-SVD of $\mathcal{X}_0$ as defined in  {Definition~\ref{def:tsvd}}.

\begin{remark} In closely related work \cite{qi2021t}, the authors considered low-tubal-rank tensor approximation from noiseless sketches, adapted and applied single-pass matrix sketching algorithms from  \cite{tropp_sketching}, while our setting here is to recover low-tubal-rank tensors from noisy sketches. The results in \cite{qi2021t} provided expected error bounds in the noiseless case, while our work provides probabilistic error bounds in the noisy case.
\end{remark}

\section{Proof of main results}\label{sec:proof}
Our proof for robust matrix recovery, presented in Theorem \ref{thm:main}, derives an upper bound for the difference between the output $X$ and the ground truth matrix $X_0$. To accomplish this, the expression, $X-X_0 $, is decomposed into two summands. The first summand, which is responsible for the first error term in Theorem \ref{thm:main}, depends on $\tilde{Z}$ (and not on $Z$) and can be written as $PX_0$ for a projection matrix $P$. Next, we use the oblique projection matrix expression in Lemma \ref{lem:projection} to simplify the expression. We then control the error by relating it to the smallest singular value of a truncated Haar unitary matrix. Here we use the crucial fact that when $\tilde{Y}$ is full-rank, $\textnormal{Im}(P)=\ker{S}$ is uniformly distributed on the Grassmannian of all $(n-r)$-dimensional subspaces in $\mathbb C^n$. When $\tilde{Y}$ is not full rank, $\textnormal{Im}(P)$ does not have that property, and our proof technique cannot be directly applied. This part of the proof is summarized in Lemma \ref{lem:Z_1=0}. 

The second summand in the decomposition of $X-X_0$, which is responsible for the second error term in Theorem \ref{thm:main}, is due to the noise term $Z$, and is simpler to handle. For this part, a lower bound on the smallest singular value of Gaussian random matrices is utilized.

The distribution of the smallest singular value of truncated Haar unitary matrices was explicitly calculated in \cite{dumitriu2012smallest,banks2020pseudospectral}. For a more general class of random matrices (including Haar orthogonal matrices), such distribution was derived in \cite{dumitriu2012smallest,edelman2008beta,ballard2019generalized} in terms of  generalized
hypergeometric functions. By using the corresponding tail probability bound \cite[Corollary 3.4]{ballard2019generalized} for truncated Haar orthogonal matrices, our analysis can be extended to real Gaussian sketching matrices $S$ and $\tilde{S}$. Our proof techniques rely heavily on the  structure of  Gaussian matrices with i.i.d. entries. For more structured sketching matrices or sub-Gaussian sketching matrices, our current approach can not be directly applied.

\subsection{Proof of Theorem \ref{thm:exact}}
\begin{proof}
Let $X_0=U_0\Sigma_0 V_0^*$ be the SVD of $X_0$ where $U_0$ is $n_1\times r_0$, $V_0$ is $n_2\times r_0$ and $\Sigma_0$ is an invertible, diagonal $r_0\times r_0$ matrix. 
Now we can write $\overline{X}$ as
\begin{align}
    \overline{X}&=U_0\Sigma_0 V_0^* \tilde{S}^* (SU_0\Sigma_0 V_0^*\tilde{S}^*)^{\dagger} SU_0\Sigma_0 V_0^*. \notag
\end{align}
Note that since $S,\tilde{S}$ are Gaussian matrices and since $U_0$ and $V_0$ are orthonormal, $SU_0\in \mathbb C^{r\times r_0}$ and $\tilde{S}V_0\in \mathbb C^{r\times r_0}$  have linearly independent columns with probability 1 and by Property (1) in Lemma \ref{lem:pinverse},
\[ (SU_0)^{\dagger}SU_0=I, \quad V_0^*\tilde{S}^* ( V_0^*\tilde{S}^*)^{\dagger}=I.
\] So with probability 1,
\begin{align}
    \overline{X}&=U_0\Sigma_0 V_0^*\tilde{S}^* ( V_0^*\tilde{S}^*)^{\dagger}\Sigma_0^{-1} (SU_0)^{\dagger} (SU_0) \Sigma_0 V_0^*= U_0 \Sigma_0 V_0^*=X_0. \notag
\end{align}
\end{proof}

\subsection{Proof of Theorem \ref{thm:main}} 
The double sketch algorithm outputs
\[ X=\tilde{Y}^* (S\tilde{Y}^*)^{\dagger} (SX_0 + Z)=\tilde X+\tilde{Y}^*  (S\tilde{Y}^*)^{\dagger} Z \]
where 
\begin{equation}\label{Xtildedefinition}
\tilde{X} :=\tilde{Y}^*(S\tilde{Y}^*)^{\dagger}SX_0.
\end{equation}
Let $A$ be such that
\[A:=\tilde{Z}^*+X_0\tilde{S}^* \in \mathbb C^{n_1\times r}.
\]
Since $A=\tilde{Y}^*$ is of rank $r$ from the assumption in Theorem \ref{thm:main}, we denote the SVD of $A$ as
\begin{align}
    A&=U_A\Sigma_A V_A^*,  \notag \\
    U_A&\in \mathbb C^{n_1\times r}, \Sigma_A\in \mathbb C^{r\times r},  V_A\in \mathbb C^{r\times r}.\label{eq:SVD_of_A}
\end{align}
Furthermore, since $S$ is independent of  $(\tilde S, \tilde{Z})$, $SU_A$ is invertible with probability $1$. Therefore,
\begin{align}
  \tilde{Y}^*(S\tilde{Y}^*)^{\dagger} &=A(SA)^{\dagger} \notag \\
  &=(U_A\Sigma_AV_A^*)(S U_A\Sigma_AV_A^*)^{\dagger}\notag\\
  &=U_A\Sigma_A V_A^*V_A\Sigma_A^{-1}(SU_A)^{-1}\notag\\
  &=U_A(SU_A)^{-1},\label{eq:AUA}
\end{align}
where in the third equality, we use  the fact that $SU_A$ is invertible to get
\begin{equation*}
\left( S U_A\Sigma_AV_A^*\right)^{\dagger}= V_A(S U_A\Sigma_A)^\dagger=V_A\Sigma_A^{-1} (SU_A)^{-1}.
\end{equation*}
Using this notation, the output of \eqref{eq:doublesketch} simplifies to 
\[ X=\tilde X+\tilde{Y}^*  (S\tilde{Y}^*)^{\dagger}  Z = \tilde{X} + U_A(SU_A)^{-1}Z.\]
Then 
\[ X-X_0=\tilde X-X_0+\tilde{Y}^*(S\tilde{Y}^*)^{-1}Z = (\tilde{X}-X_0) + U_A(SU_A)^{-1}Z,
\]
and since our goal is to bound the approximation error, we consider 
\begin{align*}
  \gennorm{X - X_0 } &\leq \gennorm{\tilde X-X_0} + \gennorm{U_A(SU_A)^{-1} Z }.
\end{align*}
Lemma~\ref{lem:Z_1=0} allows us to bound the first term on the right-hand side of the inequality above.

\begin{lemma}\label{lem:Z_1=0}
With probability at least $1-\delta_1-\delta_2$, where $\delta_1>0$, and $ 1 \ge \delta_2>\exp(-(  \sqrt{r}-\sqrt{r_0}  )^2)$, the matrix $\tilde{X}$  defined in equation \eqref{Xtildedefinition} satisfies
\begin{align}
  \gennorm{\tilde{X}-X_0 } \leq \frac{\sqrt{r(n_1-r)} \gennorm{\tilde{Z}} }{\sqrt{\delta_1}(\sqrt{r}-\sqrt{r_0}-\sqrt{\log(1/\delta_2)})}.
  \label{eq:bound1}
\end{align}
\end{lemma}

\begin{proof}
From \eqref{Xtildedefinition} and \eqref{eq:AUA}, 
\begin{align}
    \tilde{X}-X_0&= \tilde{Y}^*(S\tilde{Y}^*)^{\dagger}SX_0-X_0  \notag \\
    &=U_A(SU_A)^{-1}SX_0-X_0 \label{eq:tildeX-X_0}\\
    &=-(I-U_A(SU_A)^{-1}S)X_0 \notag \\
    &=-PX_0, \notag
\end{align}
where we have set $P:=I-U_A(SU_A)^{-1}S$. We observe that $P$ is a projection, i.e., $P^2 =P$, which satisfies 
\begin{align*}
    \text{ker } P &=U_A,\\
    \text{Im } P&= \text{ker } S =: \tilde{V}\in \mathbb C^{n_1\times (n_1-r)} .
\end{align*}
Recall that $X_0=U_0\Sigma_0V_0^*$ is the SVD of $X_0$.
From  Lemma \ref{lem:projection},
\begin{align*}
    PX_0&=\tilde{V}(U_{A,\perp}^*\tilde{V})^{-1} U_{A,\perp}^*X_0 \\
    &=\tilde{V}(U_{A,\perp}^*\tilde V)^{-1} U_{A,\perp}^*U_0\Sigma_0V_0^*.
\end{align*}
Then we can bound
\begin{align}
   \gennorm{ \tilde X-  X_0   } &=  \gennorm{ \tilde{V}(U_{A,\perp}^*\tilde V)^{-1} U_{A,\perp}^*U_0\Sigma_0V_0^*} \notag\\
    &\leq  \specnorm{\tilde{V}} \specnorm{(U_{A,\perp}^*\tilde V)^{-1}}  \gennorm{U_{A,\perp}^*U_0\Sigma_0V_0^*}\\
    &=  \specnorm{(U_{A,\perp}^*\tilde V)^{-1}}  \gennorm{U_{A,\perp}^*U_0\Sigma_0V_0^*}.\label{eq:errorbound}
\end{align}
We focus our efforts on simplifying the second term of \eqref{eq:errorbound}. Writing $A$ in terms of the SVD of $X_0$, one obtains
\begin{equation*}
    A=\tilde{Z}^* + X_0 \tilde{S}^* = U_A \Sigma_A V_A^* = \tilde{Z}^* + U_0 \Sigma_0 V_0^* \tilde{S}^*.
\end{equation*}
Rearranging terms yields
\begin{equation*}
     U_0 \Sigma_0 V_0^* \tilde{S}^* = A- \tilde{Z}^*.
\end{equation*}
Since $\tilde{S}V_0$ is of size $r\times r_0$, $\tilde{S}$ is complex Gaussian and $V_0$ has orthonormal columns, then with probability $1$, $\tilde{S}V_0$ has linearly independent columns. Thus $V_0^*\tilde{S}^*$ has linearly independent rows and from Property (1) in Lemma \ref{lem:pinverse},
\[ (V_0^*\tilde{S}^*)(V_0^*\tilde{S}^*)^{\dagger}=I.
\]
This implies that
\begin{equation*}
    U_0 \Sigma_0 = \left( A- \tilde{Z}^* \right) \left(  V_0^* \tilde{S}^* \right)^\dagger.
\end{equation*}
We note that
\begin{align*}
\gennorm{U_{A,\perp}^*U_0\Sigma_0 V_0^*} 
&\le \gennorm{ U_{A,\perp}^*U_0\Sigma_0 } \specnorm{V_0^*}  \\
&= \gennorm{ U_{A,\perp}^*U_0\Sigma_0 }  \\
&= \gennorm{ U_{A,\perp}^*  \left( A- \tilde{Z}^* \right) \left(  V_0^* \tilde{S}^* \right)^\dagger } \\
&= \gennorm{ U_{A,\perp}^*   \tilde{Z}^* \left(  V_0^* \tilde{S}^* \right)^\dagger } \\
&\le  \gennorm{ U_{A,\perp}^*   \tilde{Z}^* }  \specnorm{  \left(  V_0^* \tilde{S}^* \right)^\dagger }\\
&= \frac{ \gennorm{ U_{A,\perp}^*   \tilde{Z}^* } }{\sigma_{\min} \left( V_0^* \tilde{S}^* \right) }\le \frac{ \gennorm{ \tilde{Z} } }{\sigma_{\min} \left( V_0^* \tilde{S}^* \right) }.
\end{align*}

Using the fact that $ \specnorm{(U_{A,\perp}^*\tilde V)^{-1}}=1/{\sigma_{\min}(U_{A,\perp}^*\tilde V)}$ for the first term of \eqref{eq:errorbound}, we then obtain the following bound:  
\begin{align}\label{eq:deterministic}
 \gennorm{ \tilde X-  X_0  } \leq \frac{ \gennorm{ \tilde{Z} }}{\sigma_{\min}(U_{A,\perp}^*\tilde V) \sigma_{\min} \left( \tilde{S}V_0 \right)},
\end{align}
which holds with probability 1.

We now derive a probabilistic bound from \eqref{eq:deterministic} using concentration inequalities from random matrix theory.
Since $\tilde{V}$ can be seen as the $n_1\times (n_1-r)$ submatrix of a Haar unitary matrix $(\tilde{V}, \tilde{V}_{\perp})$ and $S$ is independent of $(\tilde{Z},\tilde S)$, by unitary invariance property, 
\begin{align}
   M:= \begin{bmatrix}
     U_{A,\perp}^*\\U_A^* 
   \end{bmatrix}(\tilde{V}, \tilde{V}_{\perp}) \notag
\end{align}
is also a Haar unitary matrix, and $U_{A,\perp}^*\tilde{V}$ is exactly the upper left $(n_1-r)\times (n_1-r)$ corner of $M$. We can apply Lemma \ref{lem:unitary_truncation} to get 
\begin{align}
   \mathbb P\left(   \specnorm{ (U_{A,\perp}^*\tilde V)^{-1}}  \leq \frac{\sqrt{r(n_1-r)}}{\sqrt\delta_1}\right)\geq 1-\delta_1 \notag
\end{align}
for any $\delta_1>0$. Since $\tilde{S}V_0$ is distributed as an $r\times r_0$ complex Gaussian random matrix, if $r>r_0$,
by Lemma \ref{lem:Gordon}, for any $\delta_2$ with $ 1 \ge \delta_2 > \exp \left( - \left( \sqrt{r} - \sqrt{r_0} \right)^2 \right) $, with probability at least $1-\delta_2$, 
\begin{align}\label{eq:minSV0}
    \sigma_{\min}(\tilde{S}V_0)\geq \sqrt{r}-\sqrt{r_0}-\sqrt{\log(1/\delta_2)}.
\end{align}
Combining the two probability estimates, with probability at least $1-\delta_1-\delta_2$,
\begin{align}
    \gennorm{ \tilde X-  X_0 } \leq \frac{\sqrt{r(n_1-r)} \gennorm{\tilde{Z} }  }{\sqrt{\delta_1}(\sqrt{r}-\sqrt{r_0}-\sqrt{\log(1/\delta_2)})}.\notag
\end{align}
\end{proof}

For the second term, by the independence between $S$ and $\tilde{Z}$, $SU_A$ is distributed as an $r \times r $ standard complex Gaussian matrix. 
Thus, it follows from Lemma \ref{lem:square} that with probability at least $1-\epsilon$ for any $0 < \epsilon <1$,
\[ \sigma_{\min}(SU_A)\geq \sqrt{\log (1/(1-\epsilon))} r^{-1/2}.
\]
Thus, with probability at least $1-\epsilon$,
\begin{align}
  \gennorm{U_A (SU_A)^{-1} Z } & \le \specnorm{U_A } \specnorm{(SU_A)^{-1} } \gennorm{ Z  } \notag\\
  &= \frac{1}{\sigma_{\min}(SU_A)} \gennorm{ Z  } \notag\\
  &\leq \frac{\sqrt{r}}{\sqrt{\log(1/(1-\epsilon))}}\gennorm{ Z  }. 
    \label{eq:bound2}
\end{align}
Combining \eqref{eq:bound1} and \eqref{eq:bound2} yields the desired result. This finishes the proof of Theorem \ref{thm:main}.

\subsection{Proof of Corollaries}
\label{sec:proof_cor}
\begin{proof}[Proof of Corollary \ref{cor:lowrankapprox}]
 Write $X_0=X_1+E$, where $X_1$ is the best rank-$r_1$ approximation to $X_0$. 
 We obtain 
 \begin{align*}
  Y &= SX_0+Z= SX_1 + (SE+ Z), \\
  \tilde{Y}&= \tilde{S}X_0^*+\tilde{Z}= \tilde{S}X_1^* + (\tilde{S}  E^* +\tilde{Z}). 
\end{align*} 

Applying  Theorem \ref{thm:main} to $X_1$ and two error terms $(SE+Z)$ and $(\tilde{S}  E^*+\tilde{Z})$, the following error bound  holds with probability at least  $1-\delta_1-\delta_2-\epsilon$,
   \begin{align*}
       &\specnorm{X-X_0}\leq \specnorm{X - X_1} + \specnorm{ E} \\ 
       &\leq  \frac{\sqrt{r(n_1-r)} \specnorm{\tilde{S}  E^* +\tilde{Z}}}{\sqrt{\delta_1}(\sqrt{r}-\sqrt{r_1}-\sqrt{\log(1/\delta_2)})} + \frac{\sqrt{r} \specnorm{SE+Z}}{\sqrt{\log(1/(1-\epsilon))}} + \specnorm{E}\\
       &\leq  \left(\frac{\sqrt{r(n_1-r)} \specnorm{\tilde{S}}}{\sqrt{\delta_1}(\sqrt{r}-\sqrt{r_1}-\sqrt{\log(1/\delta_2)})} + \frac{\sqrt{r} \specnorm{S}}{\sqrt{\log(1/(1-\epsilon))}} +1 \right)\specnorm{E}\\
       &+\frac{\sqrt{r(n_1-r)} \specnorm{\tilde{Z}}}{\sqrt{\delta_1}(\sqrt{r}-\sqrt{r_1}-\sqrt{\log(1/\delta_2)})} + \frac{\sqrt{r} \specnorm{Z}}{\sqrt{\log(1/(1-\epsilon))}}\\
         &=  \left(\frac{\sqrt{r(n_1-r)} \specnorm{\tilde{S}}}{\sqrt{\delta_1}(\sqrt{r}-\sqrt{r_1}-\sqrt{\log(1/\delta_2)})} + \frac{\sqrt{r} \specnorm{S}}{\sqrt{\log(1/(1-\epsilon))}} +1 \right) \sigma_{r_1+1} \left(X_0 \right)\\
         &+\frac{\sqrt{r(n_1-r)} \specnorm{\tilde{Z}}}{\sqrt{\delta_1}(\sqrt{r}-\sqrt{r_1}-\sqrt{\log(1/\delta_2)})} + \frac{\sqrt{r} \specnorm{Z}}{\sqrt{\log(1/(1-\epsilon))}}.
   \end{align*}
From  concentration of the operator norm of $\tilde{S}$ and $S$ in Lemma \ref{lem:Gordon}, with probability $1-2\delta_2$, we have that
\[ \specnorm{S}\leq \sqrt{r}+\sqrt{n_1}+\sqrt{\log(1/\delta_2)}, \quad   \specnorm {\tilde S}\leq \sqrt{r}+\sqrt{n_2}+\sqrt{\log(1/\delta_2)}.
\]
We then obtain the desired bound.
  \end{proof}

\begin{proof}[Proof of Corollary \ref{cor:tproduct_sketch}]
Consider the double sketch tensor approach described in \eqref{eq:tensordoublesketch}: 
\begin{align*}
\mathcal{Y} &= \mathcal{S} * \mathcal{X}_0 + \mathcal{Z}, \,\,\, \mathcal{\tilde{Y}} = \mathcal{\tilde{S}} * \mathcal{X}_0^* + \mathcal{\tilde{Z}}
\end{align*}
where $\mathcal{S}_1 = S$, $\mathcal{\tilde{S}}_1 = \tilde{S}$ and $\mathcal{S}_k =\textbf{0}$, $ \mathcal{\tilde{S}}_k = \textbf{0}$ for all $k \in \{2, \dots, n_3\}$. With this construction, after performing a mode-3 {FFT} \eqref{eq:mode3fft} on the tensors, the measurements $\widehat{\mathcal{Y}}$ and  $\widehat{\tilde{\mathcal{Y}}}$ can be decomposed into $n_3$ low-rank matrix sketches:
\begin{align}
\widehat{\mathcal{Y}}_i &= \widehat{\mathcal{S}}_i \widehat{\mathcal{X}_0}_i + \widehat{\mathcal{Z}}_i, \,\,\, \widehat{\mathcal{\tilde{Y}}}_i = \widehat{\mathcal{\tilde{S}}}_i  \widehat{\mathcal{X}_0^*}_i + \widehat{\mathcal{\tilde{Z}}}_i, \,\,\,\, i \in [n_3].
\label{eq:tensordecomp}
\end{align}
Notice that it follows from the definition of mode-3 Fourier transformation in \eqref{eq:mode3fft}, $\widehat{\mathcal S_i}$ and $\widehat{\tilde{\mathcal S_i}}$ are complex Gaussian random matrices with independent entries by the unitary invariance of the complex Gaussian distribution. 
Thus, the results of Theorem~\ref{thm:exact} and Theorem~\ref{thm:main} can be applied to produce recovery guarantees for the double sketch tensor approach. In particular, when $\mathcal{Z} =\mathbf{0} $, $\tilde{\mathcal{Z}} =\mathbf{0}$, by Theorem~\ref{thm:exact}, each of the $n_3$ transformed frontal slices $\mathcal{X}_i$  can be recovered with probability 1 from the sketches given in \eqref{eq:tensordecomp}. 

For general noise and when $r_0 < r < n_1$, we start by noting that
\begin{align*}
    \|\mathcal{X} - \mathcal{X}_0 \|_F^2 &= \sum_{k=1}^{n_3} \|\mathcal{X}_{k} - (\mathcal{X}_{0})_k \|_F^2 \\ 
    & \leq \sum_{k=1}^{n_3} \left( 2\|\widehat{\mathcal{X}}_{k} - \widetilde{\widehat{\mathcal{X}}_{k}} \|_F^2 + 2\|\widetilde{\widehat{\mathcal{X}}_{k}} - (\widehat{\mathcal{X}}_{0})_k \|_F^2\right),
\end{align*}
where $\widehat{\mathcal{X}}_{k}, (\widehat{\mathcal{X}_{0}})_k \in \mathbb{C}^{n_1 \times n_2}$ are $k^{th}$ frontal slices of the mode-3 FFT of $\mathcal{X}$ and $\mathcal{X}_0$, and $\widetilde{\widehat{\mathcal{X}}_{k}}$ is the recovered frontal slice  {when $\mathcal{Z} = \mathbf{0}$ and $\mathcal{\tilde{Z}} = \mathbf{0}$}. For the first term, we can use the same argument as in the proof of Theorem~\ref{thm:main} and invoke Lemma~\ref{lem:Z_1=0}, and for the second term, we apply \eqref{eq:bound2}.
Thus, the approximation error is:
\begin{align}
    \|\mathcal{X} - \mathcal{X}_0 \|_F^2
    & \leq \sum_{k=1}^{n_3} \left(2\|\widehat{\mathcal{X}}_{k} - \widetilde{\widehat{\mathcal{X}}_{k}} \|_F^2 + 2\|\widetilde{\widehat{\mathcal{X}}_{k}} - (\widehat{\mathcal{X}}_{0})_k \|_F^2\right) \nonumber \\
    & \leq  \sum_{k=1}^{n_3} \left( \frac{2{r(n_1-r)}\|\widehat{\tilde{\mathcal{Z}_k}}\|_F^2}{{\delta_1}(  \sqrt{r}-\sqrt{r_0}  -\sqrt{\log(1/\delta_2)})^2} + \frac{{2r}\|\widehat{\mathcal{Z}_k}\|_F^2}{{\log(1/(1-\epsilon))}} \right) \nonumber \\ 
     & = \frac{2{r(n_1-r)}\|{\tilde{\mathcal{Z}}}\|_F^2}{{\delta_1}( \sqrt{r}-\sqrt{r_0}  -\sqrt{\log(1/\delta_2)})^2}  + \frac{2{r}\|{\mathcal{Z}}\|_F^2}{{\log(1/(1-\epsilon))}}. \label{eq:tensorrobust}
\end{align}

We now prove the low-tubal-rank approximation case, which follows the proof of Corollary~\ref{cor:lowrankapprox}. Let $r_1$ be an integer and define the error tensor $\mathcal{E}$ such that $\mathcal{X}_0 = \mathcal{X}_1 + \mathcal{E}$ where $\mathcal{X}_1$ is the best rank $r_1$ approximation of $\mathcal{X}_0$. Thus, the sketches can be written as 
\begin{align*}
    \mathcal{Y} &= \mathcal{S} * \mathcal{X}_0 + \mathcal{Z} = \mathcal{S} *\mathcal{X}_1 + (\mathcal{S} * \mathcal{E} + \mathcal{Z}) \\ 
    \mathcal{\widetilde{Y}} &= \mathcal{\widetilde{S}} * \mathcal{X}_0^* + \mathcal{\widetilde{Z}} = \mathcal{\widetilde{S}} * \mathcal{X}_1^* + (\mathcal{\widetilde{S}} * \mathcal{E}^* + \mathcal{\widetilde{Z}}), 
\end{align*}
and thus we can interpret $\mathcal{Y}$ and $\mathcal{\tilde{Y}}$ as sketches of $\mathcal{X}_1$.
If $r_1 < r < n_1$ and for all $k \in [n_3]$, $\hat{\tilde{\mathcal{Y}}}_k$ is of rank $r$, we have that 
\begin{align}
    \| \mathcal{X} - \mathcal{X}_0 \|^2_F &\leq 2\|\mathcal{X} - \mathcal{X}_1 \|^2_F + 2 \| \mathcal{E}\|^2_F  \nonumber\\ 
    &\leq \frac{4{r(n_1-r)}\|\mathcal{\widetilde{S}} * \mathcal{E}^* + \mathcal{\widetilde{Z}}\|_F^2}{{\delta_1}( \sqrt{r}-\sqrt{r_1}  -\sqrt{\log(1/\delta_2)})^2}  + \frac{4{r}\|\mathcal{S} * \mathcal{E} + \mathcal{Z}\|_F^2}{{\log(1/(1-\epsilon))}} + 2 \| \mathcal{E}\|^2_F, \label{eq:lowranktensor}
\end{align}
with probability at least $1 - (\delta_1 - \delta_2 - \epsilon)n_3$, where the second inequality uses \eqref{eq:tensorrobust}. We can further decompose the term: 
\begin{align*}
    \|\mathcal{S} * \mathcal{E} + \mathcal{Z}\|_F^2 & \leq 2\|\mathcal{S}* \mathcal{E}\|_F^2 + 2\|\mathcal{Z}\|_F^2 \\
    & = 2 \sum_{k=1}^{n_3} \| \widehat{\mathcal{S}}_k \widehat{\mathcal{E}}_k\|^2_F + 2\|\mathcal{Z}\|_F^2 \\ 
    & \leq  2 \sum_{k=1}^{n_3} \| \widehat{\mathcal{S}}_k \|^2_{2\rightarrow 2} \| \widehat{\mathcal{E}}_k\|^2_F + 2\|\mathcal{Z}\|_F^2  \\
    & =  2 \| S \|^2_{2\rightarrow 2} \| {\mathcal{E}}\|^2_F + 2\|\mathcal{Z}\|_F^2~,
    \end{align*}
    where we use the fact that $\widehat{\mathcal{S}}_k = S$ for all $k$ in the last equality.
    By Lemma~\ref{lem:Gordon}, we have that with probability at least $1-\delta_2$,
    \begin{align*}
   \|\mathcal{S} * \mathcal{E} + \mathcal{Z}\|_F^2  & \leq 2 \left( \sqrt{r} + \sqrt{n_1} + \sqrt{\log(1/\delta_2)} \right){^2} \| {\mathcal{E}}\|^2_F + 2\|\mathcal{Z}\|_F^2.
\end{align*}
Similarly, for $\|\mathcal{\widetilde{S}} {*}\mathcal{E}^* + \mathcal{\widetilde{Z}}\|_F^2$, we have that: 
\begin{align*}
    \|\mathcal{\widetilde{S}}{*} \mathcal{E}^* + \mathcal{\widetilde{Z}}\|_F^2 & \leq 2\left(\sqrt{r} + \sqrt{n_2} + \sqrt{\log(1/\delta_2)}\right){^2} \| \mathcal{E} \|^2_F + 2\| \widetilde{\mathcal{Z}}\|^2_F,
\end{align*}
with probability $1-\delta_2$.
Finally, using the upper bounds on $\|\mathcal{\widetilde{S}} {*}\mathcal{E}^* + \mathcal{\widetilde{Z}}\|_F^2$ and $\|\mathcal{S} {*}\mathcal{E} + \mathcal{Z}\|_F^2$ in \eqref{eq:lowranktensor} obtains the desired result.
\end{proof}

\section{Experiments}
\label{sec:experiments}

We complement our theoretical findings with numerical experiments examining the empirical behavior of the approximation error using the double sketch algorithm for recovering low-rank matrices from {noisy} sketches\footnote{Code is available at: \url{https://github.com/anna-math/doublesketch}}. {Unless otherwise stated,} we fix the dimension of the rank $r_0$ matrix  $X_0 \in \mathbb{R}^{n_1 \times n_2} $  to be  $n =n_1 = n_2 =100$, 
initialized by multiplying two i.i.d. {Gaussian matrices} of size $100 \times r_0$ and $r_0 \times 100$. (The singular value distribution of the product of two rectangular random Gaussian matrices can be found in \cite{burda2010eigenvalues}. Notably, this is not the Marchenko-Pastur law \cite{bai2010spectral}.{)}
Empirical approximation error represents the median error over 50 trials where the sketching matrices change with each trial. The entries of the error matrices (or tensors) $Z$ and $\tilde{Z}$ ($\mathcal{Z}$ and $\tilde{\mathcal{Z}}$) are drawn i.i.d. from a standard Gaussian distribution and then normalized to achieve the desired Frobenius norm. For these numerical experiments, all norms and errors are measured in the Frobenius norm.

Figure~\ref{fig:fig1} demonstrates the performance of the double sketch algorithm for recovering a low-rank matrix when $r_0 = 10$ and $r \in \{ r_0+1, 2r_0, n-1 \}$ (left to right). For each tuple $(r,r_0)$, we plot the approximation error while varying the noise levels $\epsilon_1 = \|Z \|_F$ and $\epsilon_2 = \| \tilde{Z}\|_F$. Note that in Figure~\ref{fig:fig1} as $r$ increases and approaches $n_1$, the impact of $\epsilon_2$ in the approximation error becomes weaker. This can be seen best in the far right plot where the choice of $\epsilon_2$ seems to have no influence on the approximation error of the double sketch algorithm. This is expected as the error bound represented in Theorem~\ref{thm:main} depends on $n-r$ and if we pick $r = n$, the error bound becomes independent of $\epsilon_2$.

\begin{figure}[ht]
    \centering
    \includegraphics[width=\textwidth]{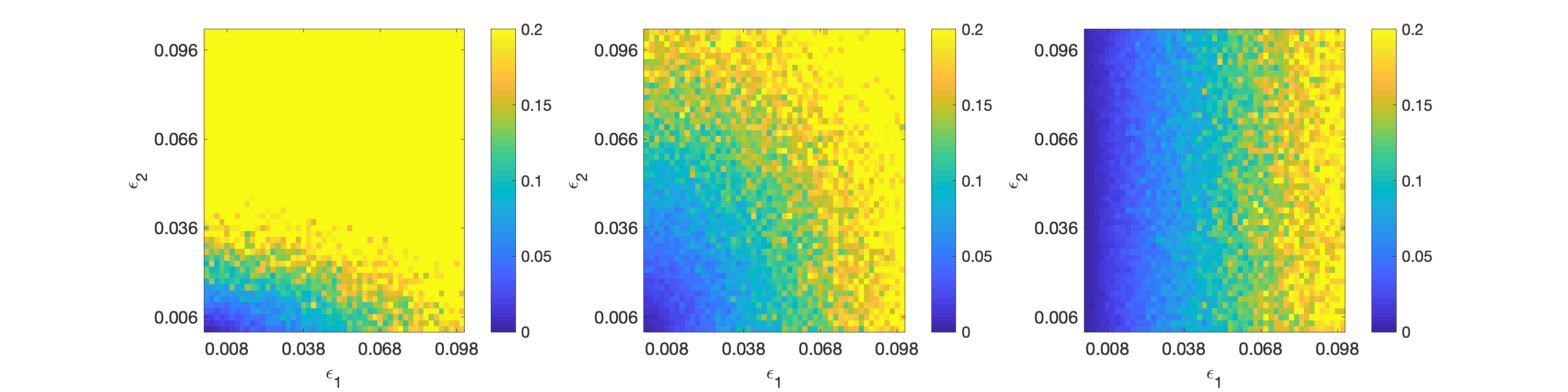}
    \caption{Numerical results demonstrating the empirical performance measured by the approximation error $\|X - X_0 \|_F$ of the double sketching algorithm for recovering a $100 \times 100$ rank $r_0 = 10$ matrix while varying $\| {{Z}}\|_F = \epsilon_1$ and $\| {\tilde{Z}}\|_F = \epsilon_2$ when (left) $r = r_0+1$ (center) $r = 2r_0$ (right) $r = n-1$. The approximation error reported is the median after 50 trials.}
    \label{fig:fig1}
\end{figure}

In our next experiment, we consider the recovery of low-tubal-rank tensors. Figure~\ref{fig:fig2} presents the empirical approximation errors obtained when using the double sketch Algorithm~\ref{eq:tensordoublesketch} to recover low-tubal-rank tensors. To initialize the underlying low-rank tensor, we multiply two tensors of dimension $n \times r_0 \times n_3$ and $r_0 \times n \times n_3$ such that $r_0 < \min\{n,n_3\}$ and we set $n=100$ in this experiment. Thus, the tubal rank of the tensor is $r_0$ and $\mathcal{X}_0 \in \mathbb{R}^{n \times n \times n_3}$. Similar to the experiment presented for the matrix case, the approximation error is the median error from 50 trials where new sketching matrices are drawn for each trial. In the tensor case, as we vary $\epsilon_1$ and $\epsilon_2$, the behavior of the approximation error is similar to the matrix case. In particular, as $r$ approaches $n$, the error bound becomes independent of $\epsilon_2$. In other words, the trend for the approximation error does not change as $\epsilon_2$ varies, as shown in the far right plot of Figure~\ref{fig:fig2}.

\begin{figure}[ht]
    \centering
    \includegraphics[width=\textwidth]{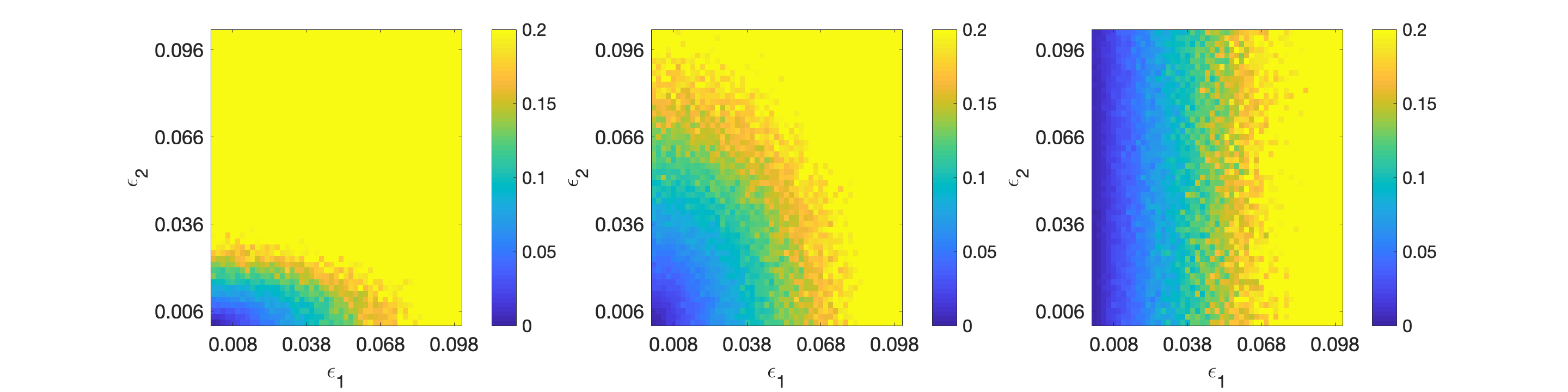}
    \caption{The Frobenius norm error $\|\mathcal{X} - \mathcal{X}_0 \|_{{F}}$ of the double sketching algorithm for recovering a $100 \times 100 \times {50}$ tensor of tubal-rank $r_0 = 10$ {when} varying $\| \mathcal{\tilde{Z}}\|_F = \epsilon_1$ and $\| \mathcal{{Z}}\|_F = \epsilon_2$ for: (left column) $r = r_0+1$ (center column) $r = 2r_0$ (right column) $r = n-1$. }
    \label{fig:fig2}  
\end{figure}

Figure~\ref{fig:fig3} presents the performance of Algorithm~\ref{alg:ds_tensor} as $n_3$ increases for {different} choices of $r$. For this experiment, $\mathcal{X}_0$ is a tubal-rank $r_0 = 10$ tensor of dimension $100 \times 100 \times n_3$ with unit Frobenius norm and {we report the median approximation error after 100 trials}. The noise level is set to $\| \mathcal{\tilde{Z}}\|_F = 0.01$ and $\| \mathcal{{Z}}\|_F = 0.01$. Here we make two observations. First, we note that as $r$ approaches $n=100$, the approximation error gets smaller. This is expected as the first term of the error bound \eqref{eq:robustrecoverytensor} depending on $\| \mathcal{\tilde{Z}}\|_F$ decreases as $r$ approaches $n$. Second, as $n_3$ increases, we see that the approximation error slowly increases with $n_3$, 
which {could be} expected since our probability bound gets worse as $n_3$ increases{, see the second statement in Corollary \ref{cor:tproduct_sketch}}.

\begin{figure}[ht]
    \centering
    \includegraphics[width=.4\textwidth]{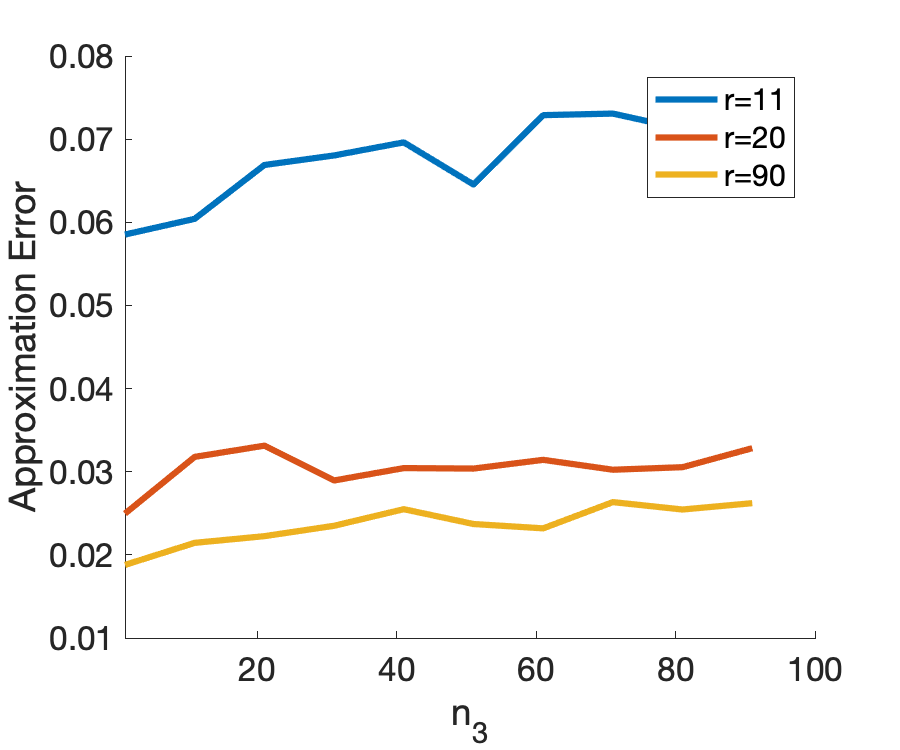}
    \caption{Frobenius norm error {$\|\mathcal{X} - \mathcal{X}_0 \|_F$} of the double sketching algorithm for recovering a $100 \times 100 \times n_3$ {tensor} while varying $n_3$. The noise level is set to $\| \mathcal{\tilde{Z}}\|_F = 0.01$ and $\| \mathcal{{Z}}\|_F = 0.01$.}
    \label{fig:fig3} 
\end{figure}

In our final experiment, we demonstrate the performance of Algorithm~\ref{alg:ds_tensor} on CT scan data. This data set contains CT scan slices of the C1 vertebrae. The image datasets used in this experiment were from the Laboratory of Human Anatomy and Embryology, University of Brussels (ULB), Belgium\footnote{Bone and joint {CT}-Scan Data \url{https://isbweb.org/data/vsj/}}. 
Here the ground-truth tensor is given by $\mathcal{X}_0 \in \mathbb{R}^{512 \times 512 \times 47}$. We pick $r = 200$ and compare the application of Algorithm~\ref{alg:ds_tensor} to the application of Algorithm~\ref{alg:ds_matrix} naively to each slice of the tensor (which we refer to as slice-wise matrix recovery). {The top row of }Figure~\ref{fig:fig4} presents a visualization of the recovery of the 20th slice of the tensor {while the bottom row visualizes the pixel-wise absolute difference between the ground truth and the recovered slice}. The first {column presents} the ground truth slice, the second {column presents} the 20th frame treated as a matrix and using Algorithm~\ref{alg:ds_matrix} for recovery, and the third {column presents} the 20th frame recovered from the double sketched tensor using Algorithm~\ref{alg:ds_tensor}. {In this experiment, we report the median error following 50 random trials of each setting.} The {median} approximation error $\|\mathcal{X} - \mathcal{X}_0  \|_F$ using Algorithm~\ref{alg:ds_tensor} is {$0.1043$} whereas the {median} approximation error applying Algorithm~\ref{alg:ds_matrix} is {$0.2545$}. Algorithm~\ref{alg:ds_tensor} has the added benefit of only requiring a single sketching matrix $S \in \mathbb{R}^{512 \times 200}$ whereas naively applying Algorithm~\ref{alg:ds_matrix} would require $n_3$ sketching matrices to be stored. Of course, it would be natural to ask whether one can simply use the same sketching matrix {for each slice in} Algorithm~\ref{alg:ds_matrix}. Doing so attained a worse {median} approximation error of {$0.3443$. We suspect this is because, with more randomness, the approximation error is better concentrated when using independent random sketch matrices for different slices.}

\begin{figure}[ht]
    \centering
    \includegraphics[width=\textwidth]{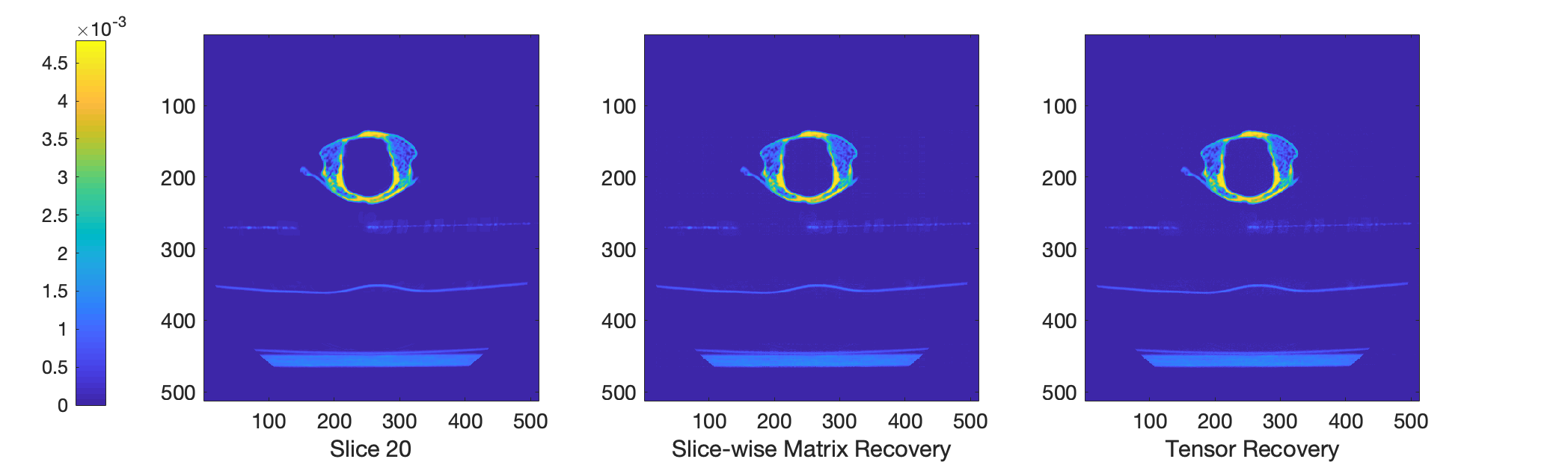}
    \includegraphics[width=\textwidth]{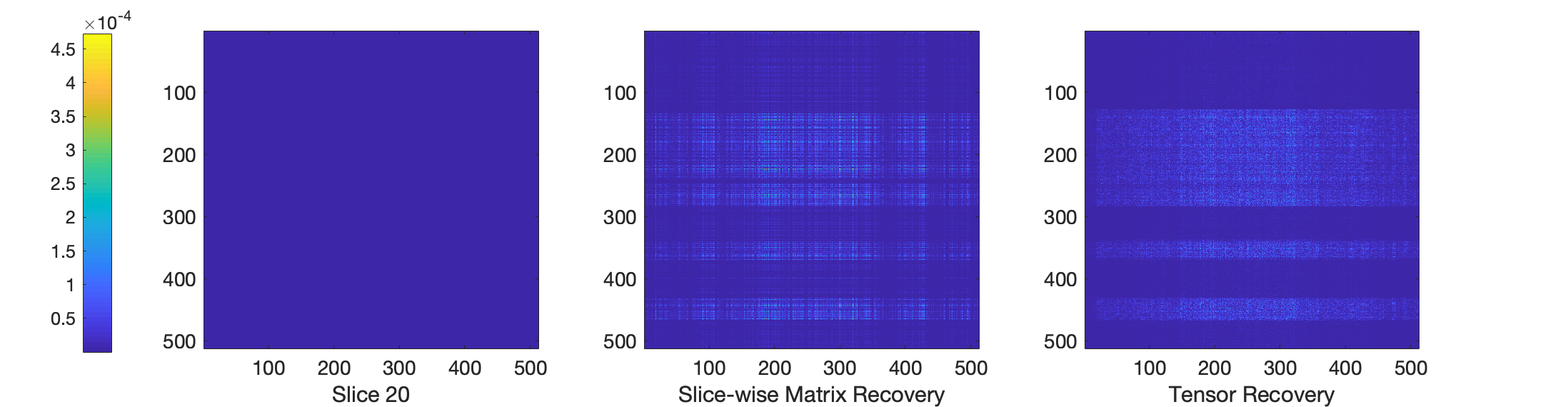}
    \caption{Performance of {Algorithm~\ref{alg:ds_tensor} compared with Algorithm~\ref{alg:ds_matrix} }applied frontal slice-wise to a tensor $\mathcal{X}_0 \in \mathbb{R}^{512 \times 512 \times 47}$ containing CT scan data of the C1 vertebrae. Here, we choose $r= 200$ and $\epsilon_1 = \epsilon_2 = 0.01$. This figure shows the 20th frontal slice for (left) ground truth tensor (center) recovered frontal slice after applying Algorithm~\ref{alg:ds_matrix} to each frontal slice (right) recovered frontal slice after applying Algorithm~\ref{alg:ds_tensor} to recover the ground truth tensor. {(top) Visualization of the recovered slice {.} (bottom) Absolute difference between ground truth and recovered slice.}}
    \label{fig:fig4} 
\end{figure}

\section{Conclusion}

In this paper, we study the problem of recovering low-rank matrices and low-tubal-rank tensors from two sketched matrices (or tensors), which are perturbed by additive noise matrices, which we denote by $Z$, respectively $\tilde{Z}$. This single-pass method{, which is based on an approach by Woolfe et al. \cite{woolfe2008fast},} has been {considered for this problem} in \cite{fazel2008compressed}{, but has not been theoretically analyzed}.
Our main results show that, without making assumptions on the distribution of the noise matrices $Z$ and $\tilde{Z}$, one can reconstruct the low-rank matrix using the method discussed in \cite{fazel2008compressed} and obtain theoretical guarantees when the sketching matrices are assumed to be Gaussian matrices. 
Furthermore, we show that this analysis is also applicable to low-tubal-rank tensor recovery. Numerical experiments corroborate our theoretical findings.

While in this paper, we have used  {the sketching method} considered in \cite{fazel2008compressed}, recent years have seen the development of sketching methods for low-rank approximation methods, which are, in general, more accurate. For example, the method in \cite{tropp_sketching2} used an additional third matrix sketch, which is called \textit{core sketch}, to improve the accuracy.
It would also be interesting to analyze the noise robustness of these more recent methods and to compare the results with the method analyzed in this paper.

Another possible extension is to analyze the truncated version of the noisy sketching method in \cite{fazel2008compressed,woolfe2008fast} for low-rank approximation, which utilizes a truncated SVD on the sketches $Y$ and $\tilde{Y}$ before performing the recovery of $X_0$~\eqref{eq:doublesketch}.

Beyond sketching for matrices and low tubal-rank tensors, it would be interesting to see whether we can analyze the reconstruction of low Tucker-rank tensors from noisy sketches.
A low Tucker-rank approximation of an order-$N$ Tucker-rank tensor using linear sketches was considered in prior work \cite{sun2020low}, where the authors used $N$ random sketching tensors to obtain $N$ factor sketches, and another $N$ random sketching matrices to obtain a core sketch. Based on those sketches, a single-pass low Tucker-rank approximation method was studied. 
In the matrix case ($N=2$), their algorithm corresponds to the one in  \cite{tropp_sketching2}.
Since the Tucker decomposition depends on tensor unfolding operations, and the sketching algorithms are applied to the unfolded tensors, it would be interesting to extend our noisy double sketching analysis to their setting (see {Algorithm 4.3} in \cite{sun2020low}).

It will also be interesting to see whether our analysis can be extended to the scenario where more structured sketching matrices such as subsampled randomized Hadamard transforms \cite{hadamard_transform} or subsampled randomized Fourier transforms (SRFT) are used. 
These more structured sketching matrices have the advantage of reducing the computational burden, and the memory footprint compared to Gaussian sketching matrices (see, e.g., \cite{tropp_sketching} for sketching algorithms for low-rank matrix approximation using SRFT matrices).
Note that while our deterministic formula for the difference between the output of the sketching algorithm and the ground truth matrix also holds in this setting, it is not known to us whether some of the involved concentration inequalities still hold in this scenario. 
We believe that this is an interesting avenue for future research.

\appendix
\section{Technical lemmas}
 
In this section, we collect a few technical lemmas, which were needed in our proofs.
\begin{lemma}[Properties of the pseudo-inverse, see, e.g., \textcolor{black}{\cite{van1996matrix}}]\label{lem:pinverse} 
The following properties hold for the pseudo-inverse.
\begin{enumerate}
\item  If $A$ has linearly independent columns, then  $A^{\dagger}A=I$. If $B$ has linearly independent rows, then $BB^{\dagger}=I$.
\item $(AB)^{\dagger}=B^{\dagger} A^{\dagger}$ if $A$ has orthonormal columns or $B$ has orthonormal rows.
    \item $(AB)^{\dagger}=B^{\dagger} A^{\dagger}$ if $A$ has linearly independent columns.
    \item If $A$ has orthonormal columns or orthonormal rows, then $A^{\dagger}=A^*$.
\end{enumerate}
\end{lemma}

\begin{lemma}[Oblique projection matrix, Theorem 2.1 in \cite{hansen2004oblique}, see also Equation (7.10.40) in \cite{meyer_linalg}]\label{lem:projection}  Let $P \in \mathbb{C}^{n \times n}$ be a projection matrix, i.e., $P^2=P$.
We note that $P$ is uniquely characterized by the subspaces $V_1$ and $V_2$ given by
\begin{align*}
    V_1&:=\textnormal{Im } P,\\
    V_2&= \textnormal{{ker} } P.
\end{align*}
and $\text{dim }V_1 + \text{dim } V_2 =n$. By some abuse of notation, we denote by $V_1$ a matrix with orthonormal columns, whose column span is equal to the subspace $V_1$.
Analogously, we denote by $V_2$ a matrix with orthonormal columns, whose column span is equal to the subspace $V_2$.
    It holds that
    \begin{align}
   P=V_1(V_{2,\perp}^*V_1)^{-1} V_{2,\perp}^*.\notag
   \end{align}
\end{lemma}

\begin{lemma}[Smallest singular value of a truncated Haar unitary matrix, Proposition C.3 in \cite{banks2020pseudospectral}]\label{lem:unitary_truncation}
Let $A$ be the upper left $(n-r)\times (n-r)$ corner of a Haar unitary matrix $U$. Then for any $\delta>0$
\begin{align}
    \mathbb P\left( \sigma_{\min}(A)\geq \frac{\sqrt{\delta}}{\sqrt{r(n-r)}} \right)\geq 1-\delta.\notag
\end{align}
\end{lemma}

\begin{lemma}[Extreme singular value of Gaussian random matrices, Corollary 5.35 in \cite{vershynin_2012}]\label{lem:Gordon}
Let $A$ be an $m\times n$ random  matrix with independent standard complex Gaussian entries and $m>n$, then for any $\delta$ with $0< \delta \le 1$, with probability at least $1-\delta$,
\[\sigma_{\min}(A)\geq \sqrt{m}-\sqrt{n}-\sqrt{\log(1/\delta)}.
\]
Similarly, with probability at least $1-\delta$,
\[\sigma_{\max}(A)\leq  \sqrt{m}+\sqrt{n}+\sqrt{\log(1/\delta)}.
\]
\end{lemma}
\begin{proof}
This lemma was proved in \cite{vershynin_2012} for real Gaussian matrices. We include a proof for complex Gaussian matrices for completeness. We have
\begin{align}
    \sigma_{\min}(A)=\min_{u\in S^{n-1}} \max_{v\in S^{m-1}} \langle A u, v\rangle, \notag
\end{align}
where $S^{n-1}$ is the unit sphere in $\mathbb C^{n}$. Let $X_{u,v}=\langle Au,v\rangle $ and $Y_{u,v}=\langle g,u\rangle+ \langle h, v\rangle$, where $g\in \mathbb C^n$, $h\in \mathbb C^{m}$ are independent standard complex Gaussian random vectors.
We find for any $x,w\in S^{n-1}$ and any $y,z\in S^{m-1}$,
\begin{align*}
    \mathbb E|X_{u,v}-X_{w,z}|^2=\|uv^*-wz^* \|_F^2\leq \|u-w\|_2^2+ \|v-z\|_2^2=\mathbb E|Y_{u,v}-Y_{w,z}|^2.
\end{align*}
From Gordon's inequality (\cite[Exercise 7.2.14]{vershynin2018high}) and \cite[Exercise 7.3.4]{vershynin2018high},
\[ \mathbb E \sigma_{\min}(A) \geq \mathbb E\|h\|_2-\mathbb E\|g\|_2.
\]

We also have $\|h\|_2^2=\sum_{i=1}^m |h_i|^2=\frac{1}{2} W$, where $W\sim \chi^2 (2m)$. So
\begin{align*}
    \mathbb E\|h\|_2-\sqrt{m}=\frac{1}{\sqrt 2}\left(\mathbb E\sqrt{W}-\sqrt{2m}\right).
\end{align*}
Note that $f(n)=\mathbb E\|g\|_2-\sqrt{n}$ is an increasing function in $n$, where $g\sim N(0,I_n)$, and $\|g\|_2^2\sim \chi^2(n)$. So for $m\geq n$,
\[  \mathbb E\|h\|_2-\sqrt{m} \geq \mathbb E\|g\|_2-\sqrt{n},
\]
which gives $\mathbb E \sigma_{\min}(A) \geq \sqrt{m}-\sqrt{n}$.
Since $\sigma_{\min}(A)$ is a Lipschitz function of $2mn$ independent real Gaussian random variables with mean zero and variance $1/2$,  the probability estimate follows from Gaussian concentration of Lipschitz functions (see \cite[Proposition 5.34]{vershynin_2012}).
The largest singular value $\sigma_{\max}(A)$ can be estimated in a similar way using the Sudakov-Fernique inequality see, e.g., \cite[Theorem 7.2.11]{vershynin2018high}.
\end{proof}

\begin{lemma}[Smallest eigenvalue of a square Gaussian random matrix, \cite{edelman1988eigenvalues} and Theorem 1.1 in \cite{tao2010random}]\label{lem:square}
Let $A$ be an $n\times n$ complex standard Gaussian random matrix. Then, for any $\epsilon>0$,
\begin{align}
   \mathbb P\left( \sigma_{\min} (A)\geq \epsilon n^{-1/2}\right)=e^{-\epsilon^2}.\notag
\end{align}
\end{lemma}

\bibliographystyle{plain}
\bibliography{ref.bib}
\end{document}